\documentclass[10pt]{article}

\usepackage{amsmath,amssymb,amsthm,color,epsfig,mathrsfs}
\usepackage{amsbsy}

\usepackage{graphicx}

\newtheorem{theorem}{Theorem}

\newtheorem{lemma}[theorem]{Lemma}

\newcounter{spslist}

\newcounter{geqncount}
    {\refstepcounter{equation}%
     \setcounter{geqncount}{\value{equation}}%
     \setcounter{equation}{0}%
  }%
    {\setcounter{equation}{\value{geqncount}}}

\newcommand{\RR}{\mathbb{R}}

\newcommand{\half}{{\textstyle\frac{1}{2}}}

\newcommand{\tphi}{\tilde\phi}
\newcommand{\tpsi}{\tilde\psi}
\newcommand{\Rn}{\mathbb{R}^n}

\newcommand{\init}{r}

\newcommand{\norm}[1]{\| #1\|}
\newcommand{\ds}{\displaystyle}

\begin{document}

\bibliographystyle{plain} 

\begin{center}
{\bf \Large  Short-Time Nonlinear Effects in the Exciton-Polariton System}
\end{center}

\vspace{0.2ex}

\begin{center}
{\scshape \large Cristi D. Guevara and Stephen P. Shipman} \\
\vspace{1ex}
{\itshape Department of Mathematics, Louisiana State University}
\end{center}

\vspace{3ex}
\centerline{\parbox{0.9\textwidth}{
{\bf Abstract.}\
In the exciton-polariton system, a linear dispersive photon field is coupled to a nonlinear exciton field.  Short-time analysis of the lossless system shows that, when the photon field is excited, the time required for that field to exhibit nonlinear effects is longer than the time required for the nonlinear Schr\"odinger equation, in which the photon field itself is nonlinear.  When the initial condition is scaled by $\epsilon^\alpha$, it is found that the relative error committed by omitting the nonlinear term in the exciton-polariton system remains within $\epsilon$ for all times up to $t=C\epsilon^\beta$, where $\beta=(1-\alpha(p-1))/(p+2)$.  This is in contrast to $\beta=1-\alpha(p-1)$ for the nonlinear Schr\"odinger equation.
}}

\vspace{3ex}
\noindent
\begin{mbox}
\end{mbox}
\vspace{3ex}

\hrule
\vspace{1.1ex}


\section{Short-time nonlinear effects in dispersive systems}\label{sec:introduction}

Both the nonlinear Schr\"odinger (NLS) equation and the exciton-polariton (EP) equations exhibit dispersion and nonlinearity.  In the NLS case, the dispersion and nonlinearity terms involve a single field $\phi(x,t)$ (with $x\in\RR^n$, $t\geq0$),
\begin{equation}\label{nls}
  i\phi_t \,=\, -\Delta\phi + g|\phi|^{p-1}\phi
\end{equation}
(the potential is taken to be zero) but in the EP case, dispersion and nonlinearity come from two different fields in a coupled system.  The dispersive term $-\Delta\phi$ involves the photon field~$\phi$, and the nonlinear term enters only through the exciton field~$\psi$ as $g|\psi|^{p-1}\psi$,
\begin{equation}\label{system0}
\begin{split}
  i\phi_t &= -\Delta\phi + \gamma\psi \\
  i\psi_t &= (\omega_0 + g|\psi|^{p-1})\psi + \gamma\phi\,
\end{split}
\end{equation}
($g$ and $\gamma$ are real numbers, and we also take $\omega_0$ to be real).  This simple distinction between the NLS and EP equations generates fundamental differences between the behaviors of the two systems.  Whereas the NLS equation is Galilean-invariant, the EP system is not.  And whereas the NLS equation admits a frequency-invariant ground state, the harmonic coherent structures of the EP system depend on frequency in a complex way, even in dimension $n=1$ as described in~\cite{KomineasShipmanVenakides2016}; NLS can be considered an approximation of EP in an appropriate sense only in a limited frequency regime~\cite[III]{KomineasShipmanVenakides2015}.

This work elucidates another behavioral difference between NLS and EP---{\em the time required for nonlinear effects to be observed}.  We take the point of view that the photon field~$\phi$ can be directly excited and observed, and that, for EP, the exciton field is hidden from the observer---it can be neither excited nor observed directly but is detected only through its effect on the photon field.
We shall consider nonlinear effects to be negligible if the relative error committed by omitting the nonlinear term is less than a small tolerance~$\epsilon$.

Let an initial condition for~$\phi$ of size $\epsilon^\alpha$ ($\alpha\geq0$) be given for both NLS and EP with $p>1$, and let the exciton field be initially absent for~EP,
\begin{equation}\label{initialcondition}
  \renewcommand{\arraystretch}{1.1}
\left.
\begin{array}{lll}
    \phi(x,0) &= \epsilon^\alpha\phi_0(x) & \text{for NLS and EP} \\
  \psi(x,0) &= 0 & \text{for EP},
\end{array}
\right.
\end{equation}
with $x\in\RR^n$.
We ask the question, up to what time is the effect of the system's nonlinearity on the photon field negligible?
More precisely, up to what time is the relative error between the solutions to the nonlinear and linear ($g=0$) systems less than $\epsilon$?  This time is a fractional power of $\epsilon$, that is, $t=C\epsilon^\beta$, as described in Theorem~\ref{thm:nls} for NLS and Theorem~\ref{thm:ep} for~EP.  For order-1 initial data ($\alpha=0$), the theorems tell us that the nonlinear effects are negligible up to time $\epsilon$ for NLS, whereas for EP, nonlinear effects are negligible up to time $\epsilon^{1/(p+2)}$; this result was reported in~\cite{GuevaraShipman2016} for $p\!=\!3$.  The relation between $\alpha$ and $\beta$ is illustrated in Figure~\ref{fig:alphabeta}.

\begin{figure}
  \centerline{
  \scalebox{0.32}{\includegraphics{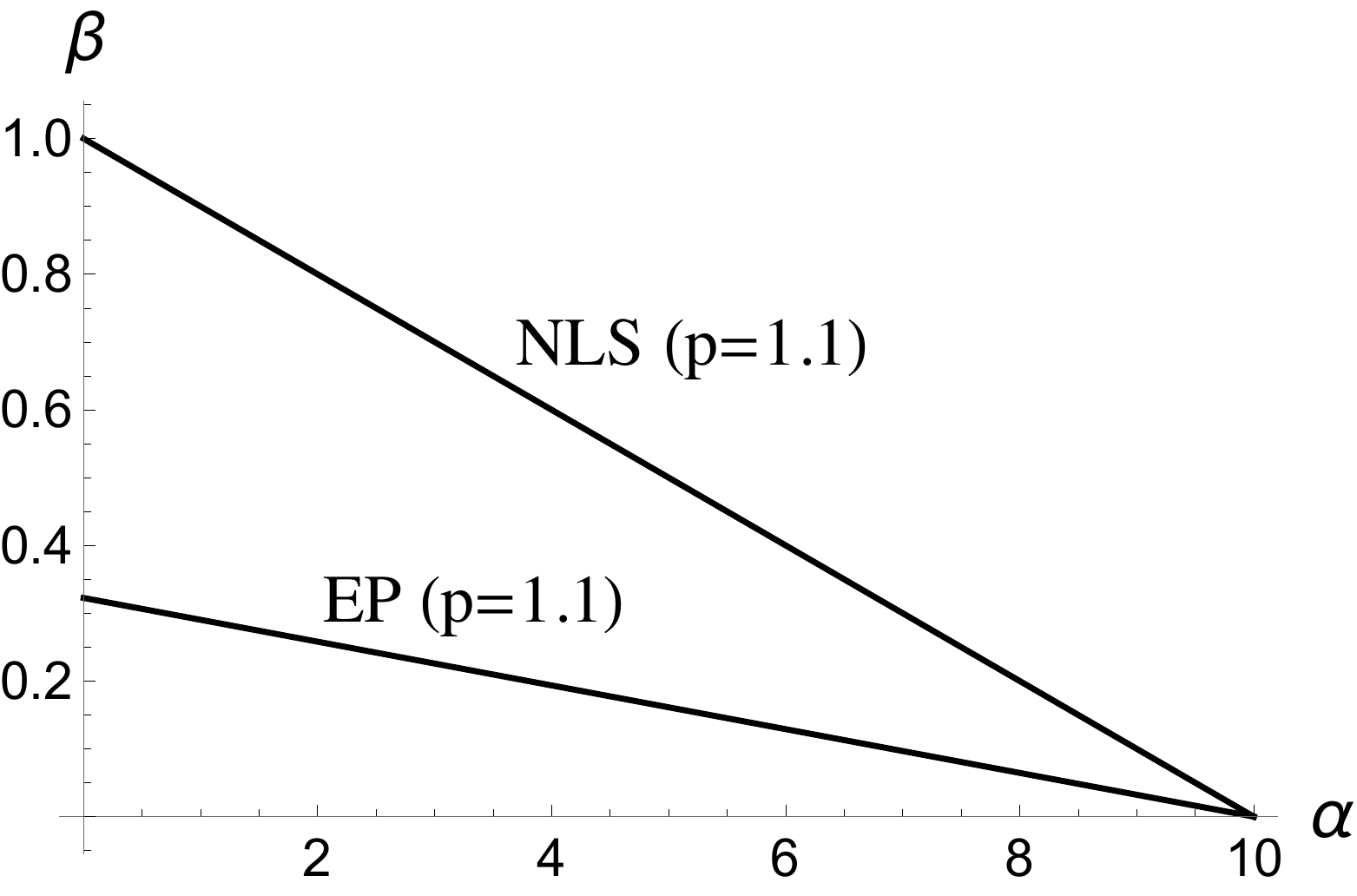}}
  \scalebox{0.32}{\includegraphics{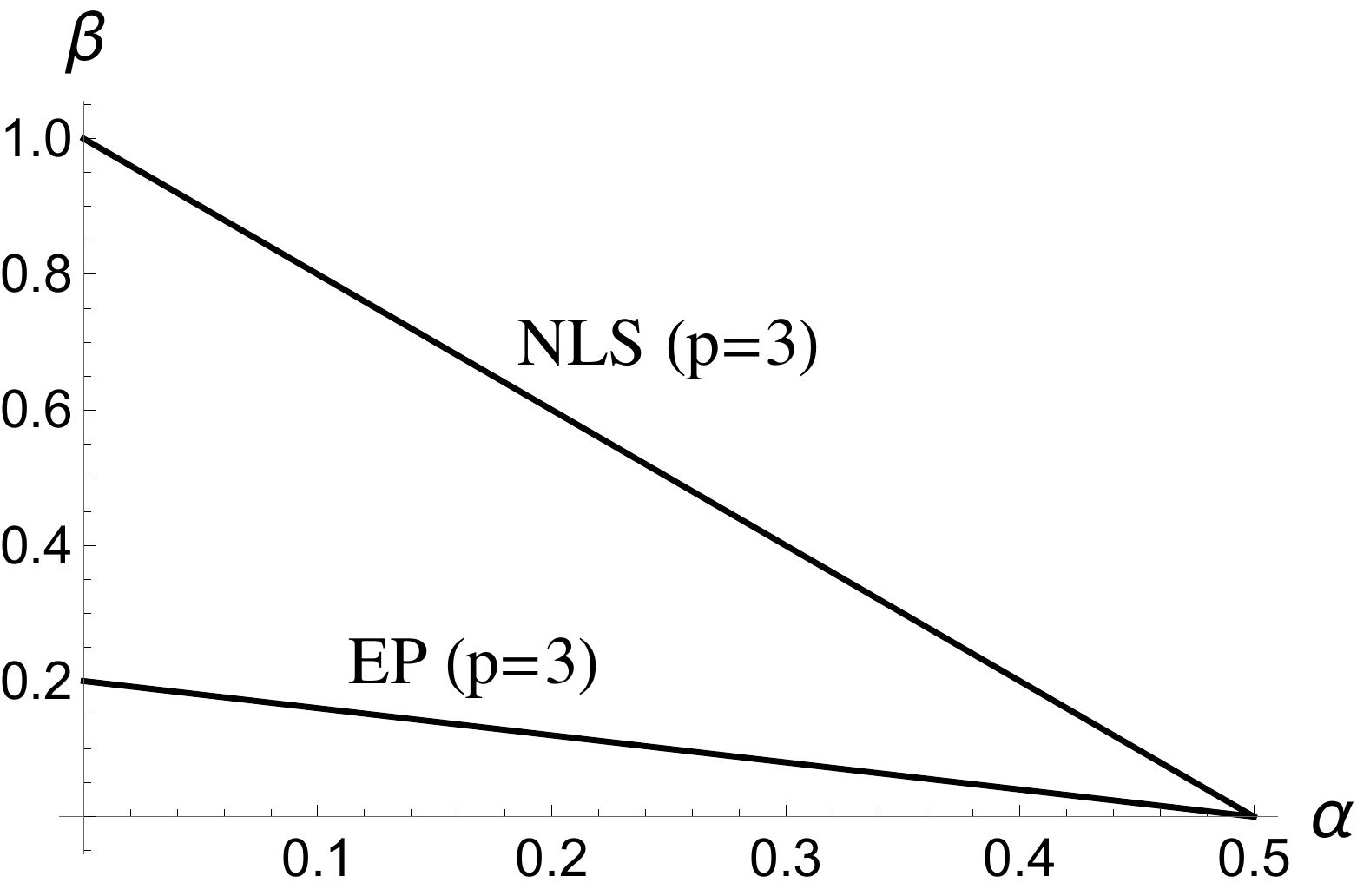}}
  \scalebox{0.32}{\includegraphics{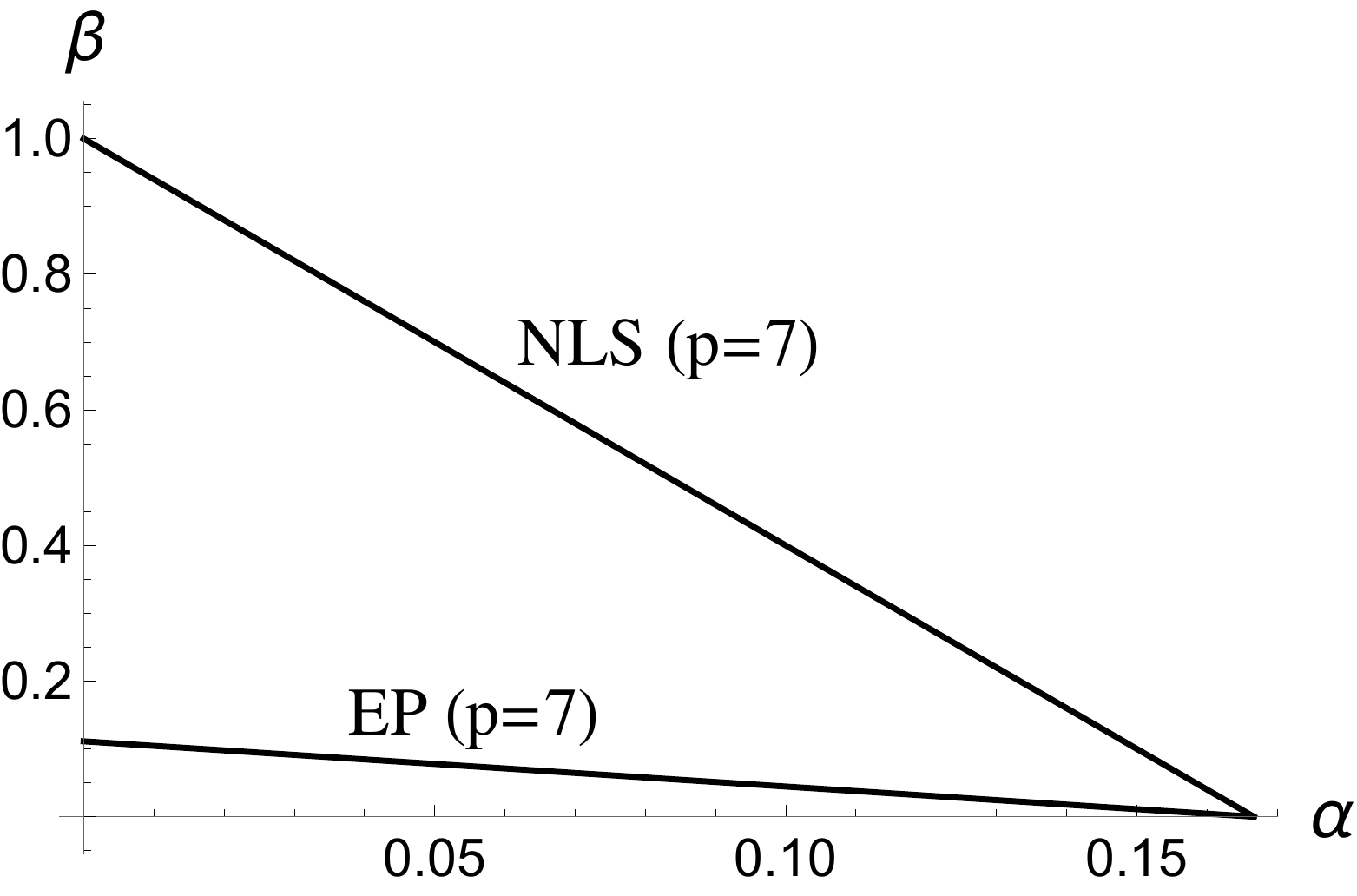}}
  }
\caption{\small 
The solution of the NLS equation~(\ref{nls}) and the EP system (\ref{system0}) with initial data $\phi(x,0)=\epsilon^\alpha\phi_0(x)$ are approximated by the solutions to the corresponding linear systems ($g=0$) within a relative error of $\epsilon$ for all times up to $t=C\epsilon^\beta$.  The relation between $\alpha$ and $\beta$ is given in Theorem~\ref{thm:nls} for the NLS equation and Theorem~\ref{thm:ep} for the EP system.
}
\label{fig:alphabeta}
\end{figure}

For the EP system, we also ask, up to what time is the influence of the exciton field on the photon field altogether negligible (not just the influence of the nonlinear term).
Initially, one expects the $\phi$ field to behave as if it were decoupled from $\psi$ and for $\psi$ to behave linearly under the influence of the photon field.  This is described by the approximate system
\begin{equation}\label{systemA}
\begin{split}
  i\phi_t &= -\Delta\phi \\
  i\psi_t &= \omega_0 \psi + \gamma\phi\,.
\end{split}
\qquad \text{(approximation A)}
\end{equation}
When the $\psi$ grows sufficiently large, one expects that it will significantly affect $\phi$ but that the effect of the nonlinearity will continue to be negligible for a while longer.  The photon will behave as if it were coupled to a linear exciton field.  This is described by the linear exciton system,
\begin{equation}\label{systemB}
\begin{split}
  i\phi_t &= -\Delta\phi + \gamma\psi \\
  i\psi_t &= \omega_0 \psi + \gamma\phi\,.
\end{split}
\qquad \text{(approximation B)}
\end{equation}

A note about the physical relevance of the EP system:  Exciton-polaritons are a quasi-particle formed when photons interact with excitons (electron-hole pairs) in a two-dimensional waveguide containing a semi-conductor layer (so $x\in\RR^2$); see \cite{CarusottoCiuti2004,Deveaud2007,MarchettiSzymanska2011}, for example.  The nonlinearity is of order $p=3$.  The physical system incorporates a frequency shift in the photon dispersion ($-\Delta \mapsto -\Delta + \omega_C$), and a free linear exciton frequency $\omega_X$ in place of $\omega_0$.  By removing a factor of $e^{-i\omega_C t}$ from both fields, one obtains (\ref{system0}) with $\omega_0=\omega_X-\omega_C$.  The physical equations also incorporate imaginary components for $\omega_C$ and $\omega_X$, representing system losses, which we take to be equal to zero.

Observe that the results in this work involve asymptotically small times and thus do not depend on~$\omega_0$.  The proofs involve only triangle-type inequalities and have not been proven to be sharp.  However, numerical simulations presented in Section~\ref{sec:numerical} indicate that results are indeed sharp.

\begin{theorem}[Short time for NLS]\label{thm:nls}
For each $\epsilon:0<\epsilon<1$, let $\phi(t)$ be the solution of the NLS equation (\ref{nls}) with $p>1$ and $\tphi(t)$ the solution of the LS equation (\ref{nls} with $g=0$), both defined in an interval $[0,T]$, having values in $H^s(\RR^n)$ for $s>n/2$, and satisfying the initial condition
\begin{equation}
  \phi(0) = \tphi(0) = \epsilon^\alpha\phi_0\,,
\end{equation}
with $\alpha\geq0$ and $\left\| \phi_0 \right\|_{H^s} = M<\infty$. 

The relative error in the approximation of $\phi$ by $\tphi$ is bounded by
\begin{equation}
  \frac{\| \tphi(t) - \phi(t) \|_{H^s}}{\left\| \phi(t) \right\|_{H^s}}
    \leq  B\,\epsilon + O(\epsilon^2)
    \qquad \text{for} \quad 0\leq t\leq C\epsilon^\beta \leq T
    \qquad (\epsilon\to0),
\end{equation}
in which 
\begin{equation}
  \renewcommand{\arraystretch}{1.1}
\left.
\begin{array}{llll}
  \beta \;=& 1-(p\!-\!1)\alpha & \text{if} & 0\leq \alpha<(p\!-\!1)^{-1} \\
  \beta \;>& 0 & \text{if} & (p\!-\!1)^{-1}\leq \alpha
\end{array}
\right.
\end{equation}
and $B = |g| K_p C M^{p-1}$ ($K_p$ is defined below after equation (\ref{Kp1})).

\end{theorem}

Proof of this theorem is given in section~\ref{sec:nls}.  The following theorem for the EP system is given in section~\ref{sec:ep}.  The special case of initial data that is not small ($\alpha=0$) is simpler and is proved in~\cite{GuevaraShipman2016} for $p=3$; the proof below subsumes $\alpha=0$.

\begin{theorem}[Short time for exciton-polariton system]\label{thm:ep}
Let real constants $C_1\geq0$, $C_2\geq0$, and $\alpha\geq0$ and a function $\phi_0\in H^s(\RR^n)$ be given.  For each $\epsilon:0<\epsilon<1$, let $(\phi(t),\psi(t))$ be the solution in an $\epsilon$-independent interval $[0,T]$ of the nonlinear polariton system (\ref{system0}) with $p>1$, and let $(\tphi(t),\tpsi(t))$ be the solution of the approximate system A (\ref{systemA}) for $0\leq t\leq C_1\epsilon^{1/2}<T$ and of the approximate system B (\ref{systemB}) for $C_1\epsilon^{1/2}\leq t < T$, with all fields being continuous functions of $t$ with values in $H^s(\RR^n)$ for $s>n/2$, and satisfying the initial condition
\begin{eqnarray}
  && \phi(0) = \tphi(0) = \epsilon^\alpha\phi_0\,, \\
  && \psi(0) = \tpsi(0) = 0,
\end{eqnarray}
with $\left\| \phi_0 \right\|_{H^s} = M<\infty$. 

The relative error in the approximation of $\phi$ by $\tphi$ is bounded by
\begin{equation}
  \renewcommand{\arraystretch}{1.1}
\left.
\begin{array}{ll}
  \displaystyle
  \frac{\| \tphi(t) - \phi(t) \|_{H^s}}{\| \phi(t) \|_{H^s}}
  \;\leq\; B_1\epsilon + O(\epsilon^q)
  & \text{ for }\;\; 0\leq t \leq C_1\epsilon^{1/2}
  \qquad (\epsilon\to0), \\
  \vspace{-1ex}\\
  \displaystyle
    \frac{\| \tphi(t) - \phi(t) \|_{H^s}}{\| \phi(t) \|_{H^s}}
  \;\leq\; B_2\epsilon + o(\epsilon)
  & \text{ for }\;\; C_1\epsilon^{1/2} \leq t \leq C_2\epsilon^\beta
  \qquad (\epsilon\to0).  
\end{array}
\right.
\end{equation}
in which\, $3/2<q=\min\left\{ 2,\, 1+p/2+\alpha(p-1) \right\}$
\begin{equation}
  \renewcommand{\arraystretch}{1.1}
\left.
\begin{array}{llll}
  \beta \;=& \big(1-(p\!-\!1)\alpha\big)/(p+2) & \text{if} & 0\leq \alpha<(p\!-\!1)^{-1} \\
  \beta \;>& 0 & \text{if} & (p\!-\!1)^{-1}\leq \alpha
\end{array}
\right.
\end{equation}
and
\begin{equation}
\begin{split}
  B_1 &= \half\gamma^2C_1^2 \\
  B_2 &= \half\gamma^2 C_1^2
               + \textstyle\frac{1}{p+2}|g|K_p\gamma^{p+1} M^{p-1} C_2^{p+2} \,\mathbf{1}_{\left\{ \alpha\leq1/2 \right\}}
\end{split}
\end{equation}
($K_p$ is defined below after equation (\ref{Kp2})).
\end{theorem}

Observe that this theorem covers the case in which only the approximate system B (\ref{systemB}) is used: the time interval in which approximate system A (\ref{systemA}) is in effect is eliminated by taking $C_1=0$.  (If $C_2=0$, then the theorem is vacuous.)

The following lemma, whose proof is elementary, will aid in the proofs of these theorems.

\begin{lemma}\label{lemma:Q} 

For $p>1$, and for sufficiently small values of $\eta$ and $\delta$, the equation
\begin{equation}
  Q(y;\eta,\delta) \,:=\, \eta\, y^p - y + \delta \,=\, 0
\end{equation}
has two real positive solutions $y=y_1(\eta,\delta)$ and $y=y_2(\eta,\delta)$.   The solution $y_1$ tends monotonically to zero with $\eta$ and $\delta$ separately and satisfies
\begin{equation}\label{y_1}
\renewcommand{\arraystretch}{1.1}
\left.
\begin{array}{lcl}
    y_1 &=& \delta \left( 1 + \eta\delta^{p-1}  + p(\eta\delta^{p-1})^2 + p^2(\eta\delta^{p-1})^3 + \dots \right)\\
  y_1^p &=& \delta^p \left( 1 + p\,\eta\delta^{p-1}  + p^2(\eta\delta^{p-1})^2 + \dots \right)\,,
\end{array}
\right.  
\end{equation}
and the solution $y_2$ tends to $\infty$ as
\begin{equation}
  y_2 \,\sim\, (\eta\delta)^{(1-p)^{-1}}\,.
\end{equation}
Furthermore, one has
\begin{equation}
  \renewcommand{\arraystretch}{1.1}
\left.
\begin{array}{ll}
  Q(y;\eta,\delta) \,>\, 0 & \text{for }\; 0\leq y<y_1 \,\text{ and }\, y_2<y \\
  Q(y;\eta,\delta) \,<\, 0 & \text{for }\; y_1<y<y_2\,. \\
\end{array}
\right.
\end{equation}
\end{lemma}

From here onward, the omission of a subscript indicating a norm will imply the norm in $H^s(\RR^n)$,
\begin{equation}
  \| \cdot \| = \| \cdot \|_{H^s}\,.
\end{equation}

\section{Short-time behavior for the NLS equation}\label{sec:nls}

The proof of Theorem~\ref{thm:nls} is uncomplicated.  First, existence theory is well established.
Kato \cite{Kato1995}, Cazenave and Weissler \cite{CaWe90}, and Ginibre and Velo \cite{GiVe84,GiVe85} proved that the NLS equation \eqref{nls} with initial data $\phi(x,0)\in H^s(\Rn)$ is locally well-posed in $H^s(\Rn)$, when either 
\begin{enumerate}
\item   $s\geq n/2$,
\item   $s<n/2$, $n\geq1$ and $p<\min\{1+4/(n-2s),1+(2s+2)/(n-2s)\}$
\item  $s<n/2$, $n=1$  and $p\leq 2/(1-2s)$.
\end{enumerate}
Additionally, Ginibre and Velo \cite{GiVe79a,GiVe79b}, Kato \cite{Kato1987}, and Tsutsumi \cite{Tsutsumi:1987aa}  proved that the NLS equation with initial data $\phi(x,0)\in H^1(\Rn)$ is locally well-posed in $H^s(\Rn)$ for $s\leq1,$  under the conditions
\begin{enumerate}
\item  $ 1<p<\infty $ for  $n\leq2$
\item   $1\leq p<\frac{n+2}{n-2}$ for $n>2$ and $g=1$
\item  $1\leq p<\frac{n+4}{n}$ for $n>2$ and  $g=-1.$
\end{enumerate}
Further, Cazenave and Weissler \cite{CaWe90} showed that for small initial data in $\dot{H}^{s}(\Rn)$, with $0\leq s<\frac{n}{2}$ and $0<p\leq \frac{n+2}{n-2},$  there exists a unique solution to the NLS equation for all times.

The solution $\phi(t)$ of the NLS equation (\ref{nls}) with initial condition $\phi(0)=\epsilon^\alpha\phi_0$, where $\phi(t)\in H^s$ and $\|\phi_0\|=\|\phi_0\|_{H^s}=M$ satisfies
\begin{equation}
  \phi(t) \,=\, \epsilon^\alpha e^{i\Delta t}\phi_0 \,+\, ig\!\int_0^t e^{i\Delta(t-\tau)} |\phi(\tau)|^{p-1}\phi(\tau)\,d\tau
\end{equation}
and hence the inequality
\begin{eqnarray} \notag
  \|\phi(t)\| &\leq& \epsilon^\alpha M \,+\, |g|\!\int_0^t \left\| |\phi(\tau)|^{p-1} \phi(\tau) \right\|\,d\tau\\
         &\leq& \epsilon^\alpha M \,+\, |g| K_p\!\int_0^t \| \phi(\tau) \|^p\,d\tau\,.  \label{Kp1}
\end{eqnarray}
The constant $K_p=K(p,s,n)$ is provided by Theorem~3.4 in \cite{LinaresPonce2014}.
Since the right-hand-side is increasing with~$t$, one has
\begin{equation}
  \sup_{\tau\leq t}\|\phi(\tau)\| \,\leq\, \epsilon^\alpha M \,+\, |g| K_p t\, \sup_{\tau\leq t}\|\phi(\tau)\|\,.
\end{equation}
This equation is equivalent to
\begin{equation}
  Q\big( \sup_{\tau\leq t}\|\phi(\tau)\|,\, |g| K_p t,\, \epsilon^\alpha M \big) \,\geq\, 0\,,
\end{equation}
in which the function $Q$ is defined in Lemma~\ref{lemma:Q}.
According to the Lemma, if $\epsilon$ and $t$ are sufficiently small, this inequality is equivalent to
\begin{equation}
  \| \phi(t) \| \,\leq\, y_1 \big( |g| K_p t,\, \epsilon^\alpha M \big)
      \,=\, \epsilon^\alpha M \Big( 1 + |g|K_p M^{p-1} t \epsilon^{\alpha(p-1)} + O(t^2\epsilon^{2\alpha(p-1)}) \Big).
\end{equation}
The supremum over $\tau\leq t$ is removed because, according to the Lemma, $y_1$ is increasing with~$t$ for fixed~$\epsilon$.

\smallskip
Let $\tilde\phi$ satisfy the linear Schr\"odinger equation with the same initial condition as for the NLS above:
\begin{align}
  i\tilde\phi_t & =\, -\Delta\tilde\phi\,, \\
  \tilde\phi(0) & =\, \epsilon^\alpha\phi_0\,.
\end{align}
For all time, one has
\begin{equation}
  \| \tilde\phi(t) \| \,=\, \epsilon^\alpha M.
\end{equation}

The difference $\hat\phi = \tilde\phi-\phi$ between the solution to the linear equation and the solution to the nonlinear one satisfies
\begin{equation}
\begin{split}
  i\hat\phi_t & =\, -\Delta\hat\phi - g|\phi|^{p-1}\phi\,, \\
  \hat\phi(0) & =\, 0\,.
\end{split}
\end{equation}
The solution is
\begin{equation}
  \hat\phi(t) \,=\, i g\! \int_0^t e^{i\Delta(t-\tau)} \left| \phi(\tau) \right|^{p-1}\!\phi(\tau)\,d\tau
\end{equation}
and satisfies the bound
\begin{equation}
\begin{split}
  \| \hat\phi(t) \| &\;\leq\; |g|\! \int_0^t \big\| \left| \phi(\tau) \right|^{p-1}\!\phi(\tau) \big\| \\
       &\;\leq\; |g|K_p\! \int_0^t \| \phi(\tau) \|^p \,d\tau  \\
       &\;\leq\; |g|K_p\! \int_0^t y_1 \big( |g| K_p \tau,\, \epsilon^\alpha M \big)^p \,d\tau\,.
\end{split}
\end{equation}
The expansion of $y_1^p$ in the Lemma leads to the bound
\begin{equation}
  \| \hat\phi(t) \| \,\leq\, |g| K_p M \epsilon^\alpha \left( M^{p-1}  t\epsilon^{\alpha(p-1)} + t\,O\big(\epsilon^{2\alpha(p-1)}\big) \right).
\end{equation}

Now imposing the assumption that $t\leq C\epsilon^\beta$ leads to the bound
\begin{equation}
  \| \hat\phi(t) \| \,\leq\, |g| K_p M \epsilon^\alpha \left( CM^{p-1}\epsilon^{\alpha(p-1)+\beta} + O\big( \epsilon^{2(\alpha(p-1)+\beta)} \big) \right).
\end{equation}
Using this together with $\| \tilde\phi(t) \| \,=\, \epsilon^\alpha M$ yields
\begin{equation}
\begin{split}
  \| \phi(t) \| &\,\geq\, \| \tilde\phi(t) \| - \| \hat\phi(t) \| \\
                  &\,\geq\, M \epsilon^\alpha - \| \hat\phi(t) \| \\
                  &\,\geq\, M \epsilon^\alpha - |g| K_p M \epsilon^\alpha \left( CM^{p-1}\epsilon^{\alpha(p-1)+\beta} + O\big( \epsilon^{2(\alpha(p-1)+\beta)} \big) \right)\,. \\
                  &\,\geq\, M \epsilon^\alpha  \left( 1 - |g| K_p CM^{p-1}\epsilon^{\alpha(p-1)+\beta} + O\big( \epsilon^{2(\alpha(p-1)+\beta)} \big) \right)\,.
\end{split}
\end{equation}
The relative error in the solution committed in omitting the nonlinear term from the NLS equation is therefore bounded~by
\begin{equation}
\begin{split}
  \frac{\| \hat\phi(t) \|}{\| \phi(t) \|} &\,\leq\,
      \frac{|g| K_p CM^{p-1}\epsilon^{\alpha(p-1)+\beta} + O\big( \epsilon^{2(\alpha(p-1)+\beta)} \big)}{\, 1 - |g| K_p CM^{p-1}\epsilon^{\alpha(p-1)+\beta} + O\big( \epsilon^{2(\alpha(p-1)+\beta)} \big)}\\
      &\,=\, |g| K_p CM^{p-1}\epsilon^{\alpha(p-1)+\beta} + O\big( \epsilon^{2(\alpha(p-1)+\beta)} \big)\,.
\end{split}
\end{equation}
Using the value of $\beta$ in Theorem~\ref{thm:nls}, one obtains
\begin{equation}
  \frac{\| \hat\phi(t) \|}{\| \phi(t) \|} \,\leq\, |g| K_p CM^{p-1}\epsilon + O(\epsilon^2)
  \qquad
  (t\leq C\epsilon^\beta)\,.
\end{equation}

\section{Short-time behavior for the polariton equations}\label{sec:ep}

Short-time analysis of exciton-polariton systems (\ref{system0},\ref{systemA},\ref{systemB}) establishes bounds on the solutions of (\ref{system0}) depending on time and $\epsilon$ and bounds on the errors committed in making the approximations (\ref{systemA},\ref{systemB}).  The results are used to prove Theorem~\ref{thm:ep}.

\subsection{Existence of solutions}\label{sec:existenceep}

Existence and uniqueness of solutions of the exciton-polariton system for short time follow from a standard contraction argument.

\begin{theorem}\label{thm:existence}
  Given $0<\init<1$, $N>0$, and  $\phi_0\in H^s(\RR^n)$ with $s>n/2$, such that $\left\| \phi_0 \right\|_{s}\leq\init N$, there exists a unique solution to the polariton equations \eqref{system0} subject to $\left\| \phi \right\|_{C^{}(I,H^s(\RR^n))}\leq N$ and $\left\| \psi \right\|_{C^{}(I,H^s(\RR^n))}\leq N$ defined for $t\in[0, T]$, where
$  T = \dfrac{1-\init}{2\gamma+|g| \tilde K N^2} $ for some constant $\tilde K$.
\end{theorem}

\begin{proof}
Write $u=(\phi, \psi)^t$, and consider the space 
\begin{align*}
E_{N,\init}=
\left\{u\in C^{}(I,H^s(\RR^n))\; : \; 
 \|u\|_{C^{}(I,H^s(\RR^n))}\leq N,\;
\|u_0\|_s\leq \init N \,\right\},
\end{align*}
with $I=[0,T]$,
equipped with the distance $d\left(u_1-u_2\right)=\|u_1-u_2\|_{C^{}(I,H^s(\RR^n))}.$  $(E_{N,\init},d)$ is a complete metric space. 
Define a mapping $\Phi: E_{N,\init} \to E_{N,\init}$~by
 \begin{align*}
 \Phi(u)(t)=
 \left(\begin{array}{c}\ds e^{it\Delta}\phi_0(x)-i\gamma\ds \int_0^te^{i({t-\tau})\Delta}\psi(\tau)d\tau 
 \\\ds-i\int_0^te^{-i\omega_0 (t-\tau)}\left(g |\psi|^2\psi(\tau) +\gamma \phi(\tau)\right)d\tau\end{array}\right).
 \end{align*}
Minkowski inequalities and the fact that $e^{it\Delta}$ is an isometry in $H^s$ yields
\begin{align*}
\norm{\Phi(u)(t)}&\leq
\norm{\phi_0}+\gamma T \Big(\sup_{\tau\leq T}\norm{\psi}+\sup_{\tau\leq T}\norm{\phi}\Big)+|g|KT \sup_{\tau\leq T} \norm{\psi}^3 \\
&\leq \norm{\phi_0}+NT \left(2\gamma +|g|K N^2\right).
\end{align*}
The constant $K$ is guaranteed by \cite[Theorem~3.4]{LinaresPonce2014} and \cite[Theorem~4.39]{Adams-Fournier:2003aa}; it relies on the algebra property of $H^s(\RR^n)$ for $s>n/2$.  For a different constant $K'$, one obtains
\begin{align*}
&\norm{\Phi(u_1)-\Phi(u_2)}\leq
 T (2\gamma   +|g| K'\!N^2)\Big(\sup_{\tau\leq T} \norm{\psi_1\!-\psi_2}+ \sup_{\tau\leq T}  \norm{\phi_1\!-\phi_2}\!\Big).
\end{align*}
Set $\tilde K = \max\left\{ K, K' \right\}$. 
Since $T\!\left(2\gamma +|g| \tilde K N^2\right)=1-\init<1$, $\Phi$ is a contraction of
$(E_{N,\init}, d)$ and thus it has a unique fixed point, which, by the definition of $\Phi$, satisfies the exciton-polariton system.  Uniqueness of the solution in $C(I,H^s(\RR^n))$ follows from Gronwall's Lemma.
\end{proof}

\subsection{Short-time bounds for solutions}\label{sec:system0}

Let $(\phi,\psi)$ be a solution of the polariton equations provided by the existence theorem, with initial conditions $\phi(0) = \epsilon^\alpha\phi_0\in H^s$ and $\psi(0)=0$.  Let $N>0$ and $\init:0<\init<1$ be given, and assume the equality
\begin{equation}
  \left\| \phi_0 \right\| = \init N =: M.
\end{equation}
By the theorem, a solution with $\|\phi(t)\|<\epsilon^\alpha N$ and $\|\psi(t)\|<\epsilon^\alpha N$ exists and is unique at least up to time~$T_1$.

The solution of the polariton equations satisfies
\begin{eqnarray}
  \phi(t) &=& \epsilon^\alpha e^{it\Delta}\phi_0 \,-\, i\gamma\! \int_0^t e^{i(t-\tau)\Delta} \psi(\tau) \,d\tau \label{phiinteqn}\\
  \psi(t) &=& -i\gamma\! \int_0^t e^{-i\omega_0(t-\tau)} \phi(\tau)\,d\tau
                   -\,ig\! \int_0^t e^{-i\omega_0(t-\tau)} |\psi(\tau)|^{p-1}\psi(\tau)\,d\tau\,. \label{psiinteqn}
\end{eqnarray}
The first of these equations implies
\begin{equation}\label{phibound1}
  \epsilon^\alpha M - \,\gamma\! \int_0^t \left\| \psi(\tau) \right\|d\tau
  \;\leq\; \left\| \phi(t) \right\| \;\leq\;
  \epsilon^\alpha M + \,\gamma\! \int_0^t \left\| \psi(\tau) \right\|d\tau\,.
\end{equation}
The second equation implies
\begin{align}
  \left\| \psi(t) \right\| &\leq \gamma\!\int_0^t \left\| \phi(\tau) \right\| d\tau \,+\, |g|\!\int_0^t \left\| |\psi(\tau)|^{p-1}\psi(\tau) \right\|d\tau  \notag \\
        &\leq \gamma\!\int_0^t \left\| \phi(\tau) \right\| d\tau \,+\, |g|K_p\! \int_0^t \left\| \psi(\tau) \right\|^p d\tau \label{Kp2}  \\\
        &\leq \gamma M\epsilon^\alpha\,t \,+\, \gamma^2\! \int_0^t\!\!\int_0^\tau \left\| \psi(\sigma) \right\|d\sigma\,d\tau
                     \,+\, |g|K_p\! \int_0^t \left\| \psi(\tau) \right\|^p d\tau\,. \label{psibound1}
\end{align}
The constant $K_p=K(p,s,n)$ is provided by Theorem~3.4 in \cite{LinaresPonce2014}.
This implies the cruder estimate
\begin{equation}
  \left\| \psi(t) \right\| \;\leq\; 
  \gamma M\epsilon^\alpha t + \half \gamma^2 t^2 \sup_{\tau\leq t}\left\| \psi(\tau) \right\|
                + |g|K_p \,t\, \sup_{\tau\leq t}\left\| \psi(\tau) \right\|^p.
\end{equation}
Since the latter expression increases with $t$, one obtains
\begin{equation}\label{cubicestimate1}
  \sup_{\tau\leq t}\left\| \psi(\tau) \right\| \;\leq\; 
 \gamma M\epsilon^\alpha t + \half \gamma^2 t^2 \sup_{\tau\leq t}\left\| \psi(\tau) \right\|
                + |g|K_p \,t\, \sup_{\tau\leq t}\left\| \psi(\tau) \right\|^p.
\end{equation}
By setting
\begin{equation}
  \eta = \left( \frac{|g|K_p\,t}{1-\half\gamma^2\,t^2} \right)^p,
  \quad
  \delta = \frac{\gamma M}{|g| K_p}\epsilon^\alpha\,,
\end{equation}
and
\begin{equation}\label{hi}
  \sup_{\tau\leq t}\left\| \psi(\tau) \right\|  = \eta^{1\!/\!p} z\,,
\end{equation}
the inequality (\ref{cubicestimate1}) is equivalent to
\begin{equation}
  Q(z,\eta,\delta) \geq 0\,.
\end{equation}
From Lemma~\ref{lemma:Q}, one obtains for sufficiently small $\epsilon$ and $t$, a quantity $z_1(\eta,\delta)$ such that $z\leq z_1(\eta,\delta)$, and therefore, by (\ref{hi}),
$\left\| \psi(t) \right\| \leq \eta^{1\!/\!p}\, z_1(\eta,\delta)$.  One computes that
\begin{equation}\label{psibound2}
   \boxed{
      \left\| \psi(t) \right\|
             \;\leq\, y_*(t,\epsilon)\,,
   }
\end{equation}
in which
\begin{equation}\label{ystar}
\begin{split}
   y_*(t,\epsilon) &:=\, \eta^{1\!/\!p}\, z_1(\eta,\delta) \\
             &\;=\, \eta^{1\!/\!p}\, \delta \big(1 + \eta\delta^{p-1} + O((\eta\delta^{p-1})^2) \big) \\
             &\;=\, \gamma M\epsilon^\alpha t\big( 1 + \half\gamma^2 t^2 + O(t^4) + |g|K_p(\gamma M)^{p-1} \epsilon^{\alpha(p-1)}t^{p}  + \epsilon^{\alpha(p-1)}O(t^{p+2})  + \epsilon^{2\alpha(p-1)}O(t^{2p})  \big)\,,  \\
    y_*(t,\epsilon)^p &\;=\, (\gamma M)^p\epsilon^{p\alpha} t^p
                         \big( 1 + \textstyle\frac{p}{2}\gamma^2 t^2 + O(t^4) + p (\gamma M)^{p-1}|g|K_p   \epsilon^{\alpha(p-1)}t^p + \epsilon^{\alpha(p-1)}O(t^{p+2}) + \epsilon^{2\alpha(p-1)}O(t^{2p}) \big)\,.
\end{split}
\end{equation}
Using this in the bound (\ref{psibound1}) yields
\begin{equation}\label{psibound3}
\begin{split}
  \left\| \psi(t) \right\| &\;\leq\;
  \gamma M\epsilon^\alpha\,t \,+\, \gamma^2\!\int_0^t\!\!\int_0^\tau y_*(\sigma,\epsilon)\,d\sigma\,d\tau
                                   \,+\, |g|K_p\! \int_0^t y_*(\tau,\epsilon)^p\, d\tau\,,
\end{split}
\end{equation}
%
and the integral term in (\ref{phibound1}) can be expanded in $t$ and~$\epsilon$~by
\begin{multline}\label{phibound2}
  \gamma\int_0^t \| \psi(\tau) \| d\tau \leq
  \gamma\int_0^t y_*(\tau,\epsilon)d\tau  \\
  \,=\, \half \gamma^2 M \epsilon^\alpha t^2
     \Big(  1 + \textstyle{\frac{1}{4}}\gamma^2 t^2 + O(t^3) + \textstyle{\frac{2}{p+2}} |g|K_p \epsilon^{\alpha(p-1)}t^p + \epsilon^{\alpha(p-1)}o(t^{2p})   \Big).
\end{multline}

\subsection{Approximation by system A}\label{sec:systemA}

Let $(\tphi(t),\tpsi(t))$ be the solution of the approximate system A (\ref{systemA}) with initial conditions $\tphi(0)=\phi_0$ and $\tpsi(0)=0$.  Let $(\phi(t),\psi(t))$ be the solution to the true system, and set
\begin{eqnarray}
  \hat\phi &:=& \tphi - \phi\,, \\
  \hat\psi &:=& \tpsi - \psi\,.
\end{eqnarray}
The pair $(\hat\phi(t),\hat\psi(t))$ satisfies the system
\begin{equation}
  \renewcommand{\arraystretch}{1.1}
\left\{
\begin{array}{ccl}
  i\hat\phi_t &=& -\Delta\hat\phi + \gamma\psi(t) \\
  i\hat\psi_t &=& \omega_0\hat\psi + \gamma\hat\phi + g|\psi(t)|^{p-1}\psi(t)
\end{array}
\right.
\qquad
\renewcommand{\arraystretch}{1.1}
\left\{
\begin{array}{l}
  \hat\phi(0) = 0 \\
  \hat\psi(0) = 0
\end{array}
\right.
\end{equation}
in which the exciton part $\psi(t)$ of the solution to the true system provides forcing in both equations.
The solution to this system satisfies
\begin{eqnarray}
  \hat\phi(t) &=& i\gamma\int_0^t e^{i\Delta(t-\tau)}\psi(\tau)\,d\tau \\
  \hat\psi(t) &=& -i\gamma\int_0^t e^{-i\omega_0(t-\tau)}\hat\phi(\tau)\,d\tau
                    \,+\, ig\int_0^t e^{-i\omega_0(t-\tau)} |\psi(\tau)|^{p-1}\psi(\tau) \,d\tau\,.
\end{eqnarray}
From these equations follow the estimate
\begin{equation}\label{phihatA}
  \| \hat\phi(t) \| \,\leq\, \gamma\int_0^t \| \psi(\tau) \| d\tau
                       \,\leq\, \gamma \int_0^t y_*(\tau,\epsilon)\,d\tau
\end{equation}
and the estimate
\begin{equation}\label{psihatA}
\begin{split}
  \| \hat\psi(t) \|  &\leq \gamma\!\int_0^t \| \hat\phi(\tau) \|d\tau \,+\,
                         |g|\!\int_0^t \big\| |\psi(\tau)|^{p-1}\psi(\tau) \big\| d\tau \\
                  &\leq \gamma\!\int_0^t \| \hat\phi(\tau) \|d\tau \,+\,
                         |g|K_p\! \int_0^t \| \psi(\tau) \|^p\, d\tau \\
                  &\leq \gamma^2\!\int_0^t\!\!\int_0^\tau y_*(\sigma,\epsilon)\,d\sigma\,d\tau \,+\,
                         |g|K_p\! \int_0^t y_*(\tau,\epsilon)^p\, d\tau\,. \\
\end{split}
\end{equation}

\subsection{Approximation by system B}\label{sec:systemB}

Now let $(\tphi(t),\tpsi(t))$ be the solution of the approximate system B (\ref{systemB}) with arbitrary initial conditions.  Let $(\phi(t),\psi(t))$ be the solution to the true system, and set
\begin{eqnarray}
  \hat\phi &:=& \tphi - \phi\,, \\
  \hat\psi &:=& \tpsi - \psi\,.
\end{eqnarray}
The pair $(\hat\phi(t),\hat\psi(t))$ satisfies the system
\begin{equation}
  \renewcommand{\arraystretch}{1.1}
\left\{
\begin{array}{ccl}
  i\hat\phi_t &=& -\Delta\hat\phi + \gamma\hat\psi(t) \\
  i\hat\psi_t &=& \omega_0\hat\psi + \gamma\hat\phi + g|\psi(t)|^{p-1}\psi(t)
\end{array}
\right.
\qquad
\renewcommand{\arraystretch}{1.1}
\left\{
\begin{array}{l}
  \hat\phi(t_1) = \hat\phi_0 \\
  \hat\psi(t_1) = \hat\psi_0\,.
\end{array}
\right.
\end{equation}
The solution satisfies
\begin{eqnarray}
  \hat\phi(t) &=& e^{i\Delta t}\hat\phi_0 \,-\, i\gamma\!\int_{t_1}^t e^{i\Delta(t-\tau)}\hat\psi(\tau)\,d\tau \\
  \hat\psi(t) &=&  e^{-i\omega_0 t}\hat\psi_0 \,-\, i\gamma\!\int_{t_1}^t e^{-i\omega_0(t-\tau)}\hat\phi(\tau)\,d\tau
                    \,+\, ig\!\int_{t_1}^t e^{-i\omega_0(t-\tau)} |\psi(\tau)|^{p-1}\psi(\tau) \,d\tau\,.
\end{eqnarray}
One then obtains the estimates
\begin{equation}\label{hatphiB}
  \| \hat\phi(t) \| \;\;\leq\;\;
      \| \hat\phi_0 \| \,+\, \gamma\int_{t_1}^t \| \hat\psi(\tau) \|d\tau
\end{equation}
and
\begin{equation}\label{hatpsibound2}
\begin{split}
  \| \hat\psi(t) \| &\leq\;\; \| \hat\psi_0 \| \,+\, \gamma\!\int_{t_1}^t \| \hat\phi(\tau) \|d\tau
                                        \,+\, |g|K_p\!\int_{t_1}^t \| \psi(\tau) \|^p d\tau \\
                        &\leq\;\; \| \hat\psi_0 \| \,+\, \gamma (t-t_1) \| \hat\phi_0 \|
                           \,+\, \gamma^2\!\int_{t_1}^t\!\!\int_{t_1}^\tau \| \hat\psi(\sigma) \| d\sigma\,d\tau
                           \,+\, |g|K_p\!\int_{t_1}^t \| \psi(t) \|^p d\tau\,.
\end{split}  
\end{equation}
As before, this yields the cruder estimate
\begin{equation}
  \sup_{t_1\leq\tau\leq t}\| \hat\psi(\tau) \| \;\;=\;\;
         \| \hat\psi_0 \| \,+\, \gamma (t-t_1) \| \hat\phi_0 \|
                 \,+\, \half\gamma^2 (t-t_1)^2 \sup_{t_1\leq\tau\leq t}\| \hat\psi(\tau) \|
                 \,+\, |g|K (t-t_1) \sup_{t_1\leq\tau\leq t}\| \psi(\tau) \|^p,
\end{equation}
which results in
\begin{equation}\label{suphatpsiB}
  \sup_{t_1\leq\tau\leq t}\| \hat\psi(\tau) \|
    \;\;\leq\;\;  \big( 1 - \half\gamma^2 t^2 \big)^{-1}
        \big( \| \hat\psi_0 \| \,+\, \gamma t \| \hat\phi_0 \| \,+\, |g|K_p t\, y_*(t,\epsilon)^p \big).
\end{equation}
One can use this estimate in (\ref{hatphiB}) to obtain one for $\| \hat\phi(t) \|$.

\subsection{Proof of Theorem~\ref{thm:ep}}\label{sec:proofep}

For the solutions $(\phi,\psi)$ and $(\tphi,\tpsi)$ in the theorem, define
\begin{eqnarray}
  \hat\phi &:=& \tphi - \phi\,, \\
  \hat\psi &:=& \tpsi - \psi\,.
\end{eqnarray}
Set $t_1=C_1\epsilon^{1/2}$ and $t_2=C_2\epsilon^\beta$, with $\beta$ defined in Theorem~\ref{thm:ep}, and $M=\init N$.

Assume first that $0\leq t\leq t_1$ so that the bounds from section~\ref{sec:systemA} apply.
The bound (\ref{phihatA}) with (\ref{phibound2}) yields
\begin{equation}\label{hatphi1}
\begin{split}
  \| \hat\phi(t) \| &\leq \gamma\int_0^t y_*(\tau,\epsilon)\,d\tau
  = \half\gamma^2 M \epsilon^\alpha \big( t^2 + O(t^4) + O(t^{p+2})\epsilon^{\alpha(p-1)} \big) \\
              &\leq \half\gamma^2 M \epsilon^\alpha
                 \big( C_1^2 \epsilon + O(\epsilon^q) \big),
\end{split}
\end{equation}
in which $q = \min\left\{ 2,\, 1+p/2+\alpha(p-1) \right\}$ and thus $q>3/2$.
From (\ref{phibound1}) and (\ref{phibound2}), one obtains $\| \phi(t) \| \geq \epsilon^\alpha M - \| \hat\phi(t) \|$ and thus
\begin{equation}
   \frac{\| \hat\phi(t) \|}{\| \phi(t) \|} \;\leq\; \frac{\half\gamma^2C_1^2 \epsilon + O(\epsilon^q)}{1- \half\gamma^2 C_1^2\epsilon - O(\epsilon^q)}
   = \half\gamma^2C_1^2 \epsilon + O(\epsilon^q)\,.
\end{equation}

Now assume that $t_1 \leq t\leq t_2$ so that the bounds from section~\ref{sec:systemB} apply.
A bound on the approximate exciton field $\hat\psi$ at $t=t_1$ is obtained from (\ref{psihatA}) and (\ref{ystar}),
\begin{equation}
\begin{split}
  \| \hat\psi(t_1) \| &\;\leq\;\, \gamma^3M\epsilon^\alpha \!
     \Big[
     \textstyle\frac{1}{6}t_1^3 + \textstyle\frac{1}{40}\gamma^2 t_1^5 + O(t_1^7) \;+  \\
     & \hspace{9em} +\;  \textstyle\frac{|g|K_p(\gamma M)^{p-1}}{(p+2)(p+3)}\epsilon^{\alpha(p-1)}t_1^{p+3} + \epsilon^{2\alpha(p-1)}O(t_1^{2p+3}) + \epsilon^{\alpha(p-1)}O(t_1^{p+5})
     \Big]\,+\\
   & \;+\; |g|K_p(\gamma M)^p\epsilon^{p\alpha}
   \Big[
   \textstyle\frac{1}{p+1}t_1^{p+1} + \textstyle\frac{p\gamma^2}{2(p+3)}t_1^{p+3} + O(t_1^{p+5}) \;+  \\
     & \hspace{9em}  +\; \textstyle\frac{p(\gamma M)^{p-1}|g|K_p}{2p+1}\epsilon^{\alpha(p-1)}t_1^{2p+1} + \epsilon^{\alpha(p-1)}O(t_1^{2p+3}) + \epsilon^{2\alpha(p-1)}O(t_1^{3p+1})
     \Big].
\end{split}
\end{equation}
Putting $t_1=C_1\epsilon^{1/2}$ produces terms of various powers of $\epsilon$, and one finds
%
%
%
\begin{equation}\label{hatpsi2}
  \| \hat\psi(t_1) \| \;\leq\; \epsilon^\alpha \left( \textstyle\frac{1}{6}\gamma^3 M C_1^3\epsilon^{3/2} + \textstyle\frac{1}{40}\gamma^5 M C_1^5\epsilon^{5/2} + \textstyle\frac{1}{p+1}|g|K_p(\gamma M)^p
C_1^{p+1} \epsilon^{\alpha(p-1)+(p+1)/2} + \mathcal{E}(\epsilon)    \right).
\end{equation}
All of the remaining terms, gathered in the ``error" $\mathcal{E}(\epsilon)$ are of higher order than $\epsilon^{3/2}$, and all of them are either of higher order than $\epsilon^{5/2}$ or of higher order than $\epsilon^{\alpha (p-1)+(p+1)/2}$.
Depending on the values of $\alpha\geq0$ and $p>1$, the term of order $\epsilon^{\alpha(p-1)+(p+1)/2}$ may be the leading-order term, or it may fall between the $\epsilon^{3/2}$ and $\epsilon^{5/2}$ terms, or it may be of higher order.  In any case, one can read off the two leading terms of this series.
A bound for the approximate photon field at $t=t_1$ is provided by (\ref{hatphi1}),
\begin{equation}\label{hatphi2}
  \| \hat\phi(t_1) \|       \leq \half\gamma^2 M \epsilon^\alpha
                 \big( C_1^2 \epsilon + O(\epsilon^q) \big),
\end{equation}
in which again $q = \min\left\{ 2,\, \alpha(p-1)+p/2+1 \right\}>3/2$.
%
%

For $t_1 \leq t\leq t_2$, one obtains from (\ref{suphatpsiB}), (\ref{hatpsi2}), and (\ref{hatphi2}) a bound on $\hat\psi(t)$,
\begin{equation}\label{hatpsiB}
\begin{split}
  \|\hat\psi(t)\| \,\leq\,&\; \epsilon^\alpha \left( 1 + \half\gamma^2 t^2 + O(t^4) \right)
  \Big[ \big( a_1\epsilon^{3/2} + a_2\epsilon^{5/2} + a_3\epsilon^{\alpha(p-1)+(p+1)/2} + {\mathcal E}(\epsilon) \big) \\
     &   \,+\,  \gamma t \big( b_1\epsilon + O(\epsilon^q) \big) \,+\, \\
     &   \,+\,  d_0\epsilon^{\alpha(p-1)}t^{p+1}
                   \big( 1 + d_1t^2 + O(t^4) + d_2\epsilon^{\alpha(p-1)}t^p + \epsilon^{\alpha(p-1)}O(t^{p+2}) + \epsilon^{2\alpha(p-1)}O(t^{2p}) \big)
  \Big]\,,
\end{split}
\end{equation}
for appropriate constants, among which $d_0=|g|K_p(\gamma M)^p$.
Using this in (\ref{hatphiB}) yields a bound on $\hat\phi(t)$,
\begin{equation}
\begin{split}
  \|\hat\phi(t)\| 
          \,\leq\,&\; \half\gamma^2 M\epsilon^\alpha \big( C_1^2\epsilon + O(\epsilon^q) \big) \,+\, \\
           &\; +\, \epsilon^\alpha \gamma \Big[
                          \big( t + O(t^3) \big) \big( a_1\epsilon^{3/2} + a_2\epsilon^{5/2} + a_3\epsilon^{\alpha(p-1)+(p+1)/2} + {\mathcal E}(\epsilon) \big) \,+ \\
            &\qquad\;+\, \gamma\big( \half t^2 + O(t^4) \big) \big( b_1\epsilon + O(\epsilon^q) \big) \,+\\
            &\qquad\;+\, d_0\epsilon^{\alpha(p-1)}  \big( \textstyle\frac{1}{p+2}t^{p+2} + \textstyle\frac{d_1+\gamma^2/2}{p+4}t^{p+4} + O(t^{p+6})\,+ \\
            & \qquad\qquad +\, \textstyle\frac{d_2}{2p+2}\epsilon^{\alpha(p-1)}t^{2p+2} + \epsilon^{\alpha(p-1)}O(t^{2p+4}) + \epsilon^{2\alpha(p-1)}O(t^{3p+2}) \big)
           \Big].
\end{split}
\end{equation}
Now using $t\leq C_2\epsilon^\beta$, one obtains
\begin{equation}
\begin{split}
  \|\hat\phi(t)\| \,\leq\,&\; \gamma\epsilon^\alpha
  \Big[\, {\half\gamma MC_1^2\epsilon} + O(\epsilon^q) + a_1C_2\epsilon^{\beta+3/2} + a_3C_2\epsilon^{\beta+\alpha(p-1)+(p+1)/2} + o(\epsilon^{\beta+3/2}) \,+ \\
  &\quad +\, \half\gamma C_2^2b_1\epsilon^{1+2\beta} + o(\epsilon^{1+2\beta}) \,+\\
  &\quad +\, \boxed{\textstyle\frac{d_0}{p+2}C_2^{p+2}\epsilon^{\alpha(p-1)+\beta(p+2)}}
     \,+\, O(\epsilon^{\alpha(p-1)+\beta(p+4)}) + O(\epsilon^{\alpha(p-1)+\beta(2p+2)})
   \;\Big].
\end{split}
\end{equation}
The crucial term is the boxed one: $\beta$ must be chosen so that the exponent on $\epsilon$ does not exceed $1$.  The value of $\beta$ given in the theorem makes this exponent equal to $1$ as long as $\beta$ remains non-negative, that is, for $\alpha\leq(p-1)^{-1}$.  For $\alpha>(p-1)^{-1}$, this exponent is greater than $1$ because $\beta$ is required to be positive.  This leads to the bound
\begin{equation}\label{hatphiB2}
  \|\hat\phi(t)\| \;\leq\; \epsilon^\alpha M \left( B_2\epsilon + o(\epsilon) \right)
  \qquad (t_1\leq t\leq t_2\,,\;\; \epsilon\to0)\,,
\end{equation}
for the constant $B_2$ defined in the theorem.
The bounds (\ref{phibound1}) and (\ref{phibound2}) provide a lower bound on $\phi(t)$,
\begin{equation}
  \|\phi(t)\| \,\geq\, \epsilon^\alpha M \left( 1 - \half\gamma^2 C_2^2 \big( \epsilon^{2\beta} + \textstyle\frac{2}{p+2}|g|K_pC_2^p\epsilon + {\mathcal E'}(\epsilon) \big) \right)
  =
\renewcommand{\arraystretch}{1.1}
\left\{\!
\begin{array}{lll}
  \epsilon^\alpha M(1 + O(\epsilon)) & \text{if} & \alpha <(p-1)^{-1} \\
  \epsilon^\alpha M(1 + o(1)) & \text{if} & \alpha \geq (p-1)^{-1}\,,
\end{array}
\right.
\end{equation}
in which the error ${\mathcal E'}(\epsilon)$ is of order less than that at least one of the two terms preceding it.

Finally, one obtains
\begin{equation}
  \frac{\| \hat\phi(t) \|}{\| \phi(t) \|} \;\leq\;
  B_2\epsilon + o(\epsilon)
  \qquad (t_1\leq t\leq t_2,\; \epsilon\to0),
\end{equation}
and this completes the proof of Theorem~\ref{thm:ep}.

\section{Numerical verification}\label{sec:numerical}

Numerical simulations of the linear and nonlinear exciton-polariton systems confirm the relation
$\beta = (1-(p-1)\alpha)/(p+2)$ announced in Theorem~\ref{thm:ep}.  In fact they universally indicate that the bounds of the Theorem are sharp.
For various values of the nonlinearity power $p$ and the spatial dimension $n$, we numerically determine a relation between $\beta$ and $\alpha$ and observe that it compares excellently with the theoretical relation.

We fix a function $\phi_0(x)$ for the initial photon field and scale it by a small number $\delta$; the function $\phi_0(x)$ is taken to depend only on the radial variable $|x|$.  We simulate the solution $\phi$ of the full exciton-polariton system~(\ref{system0}), and the solution $\tilde\phi$ of the linear system~(\ref{systemB}), both with initial condition (\ref{initialcondition}) with $\delta\!=\!\epsilon^\alpha$.  The relative error is
\begin{equation}
  \rho(t;\delta) \,=\, \frac{\| \tphi(t) - \phi(t) \|_{H^s}}{\left\| \phi(t) \right\|_{H^s}}\,.
\end{equation}

The values of $\delta$ come from a set $\Delta$ that is part of a sequence of positive real numbers converging to zero.  The values of $\alpha$ come from a set $A$.  We implement the following Algorithm~A to determine $\beta$ versus $\alpha$.  The results are shown in Figure~\ref{fig:Hs} for a Sobolev degree $s>n/2$, for which Theorem~\ref{thm:ep} is proved.  Computations of $\beta$ versus $\alpha$ when $s=0$ in Figure~\ref{fig:L2} indicate that the results of the theorem hold even when functional norms are computed in~$L^2(\RR^n)$.

\bigskip
\noindent
\parbox{\textwidth}
{
\noindent
{\bf Algorithm~A}

\smallskip
\noindent
Fix a function $\phi_0(x)$, a dimension $n$, and a power $p$.

\noindent
For each $\delta\in\Delta$:

\noindent
\hspace{2em}
Simulate system~(\ref{system0}) and system~(\ref{systemB}) for $t\in[0,T]$ with
\begin{equation*}
  \renewcommand{\arraystretch}{1.1}
\left.
\begin{array}{l}
    \phi(x,0) = \delta\,\phi_0(x)\,, \\
  \psi(x,0) = 0\,.
\end{array}
\right.
\end{equation*}

\noindent
\hspace{2em}
Compute the relative error $\rho(t;\delta)$ between the two solutions.

\noindent
For each $\alpha\in A$:

\noindent
\hspace{2em}
For each $\delta\in\Delta$:

\noindent
\hspace{4em}
Determine $\epsilon$ such that $\delta=\epsilon^\alpha$.

\noindent
\hspace{4em}
Determine $t_{\alpha,\delta}$ such that $\rho(t_{\alpha,\delta};\delta)=\epsilon$.

\noindent
\hspace{2em}
Compute the linear regression for $\log t_{\alpha,\delta}$ versus $\log\epsilon$ and set the slope equal to $\beta_\alpha$.

\noindent
Plot $\beta_\alpha$ versus $\alpha$ over $\alpha\in A$ and the linear regression.
}

\begin{figure}
\scalebox{0.164}{\includegraphics{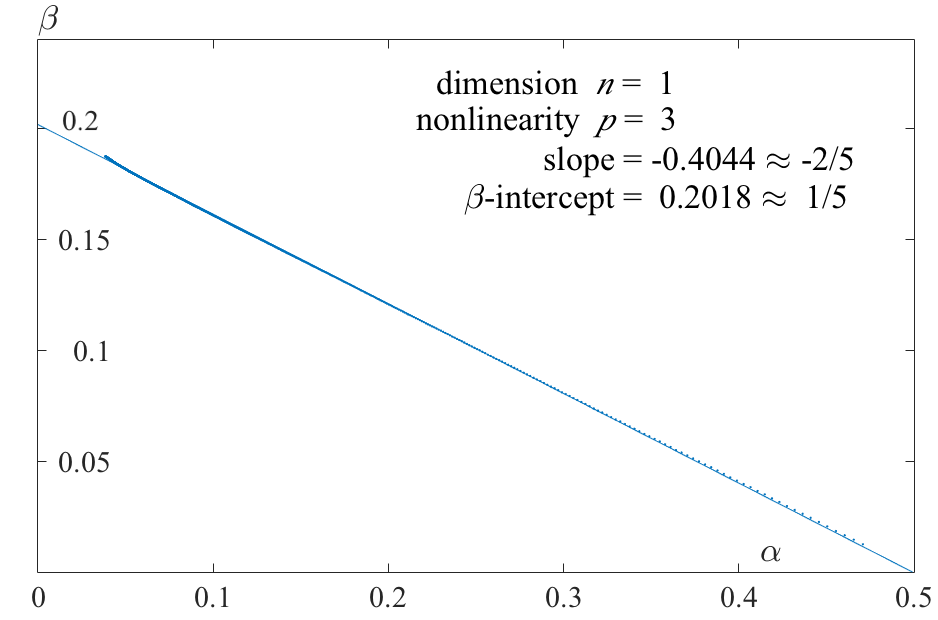}}
\hfill
\scalebox{0.164}{\includegraphics{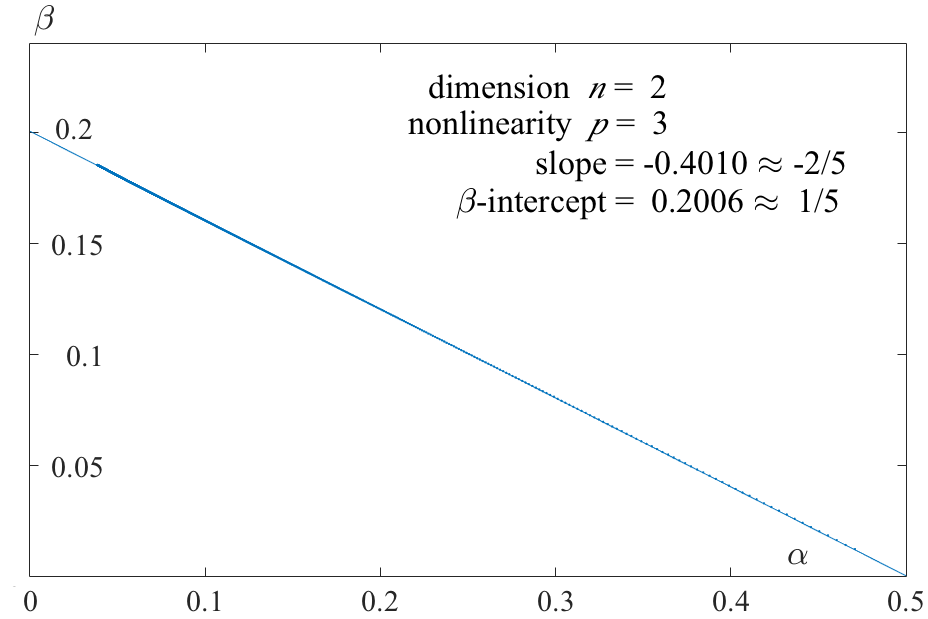}}
\hfill
\scalebox{0.164}{\includegraphics{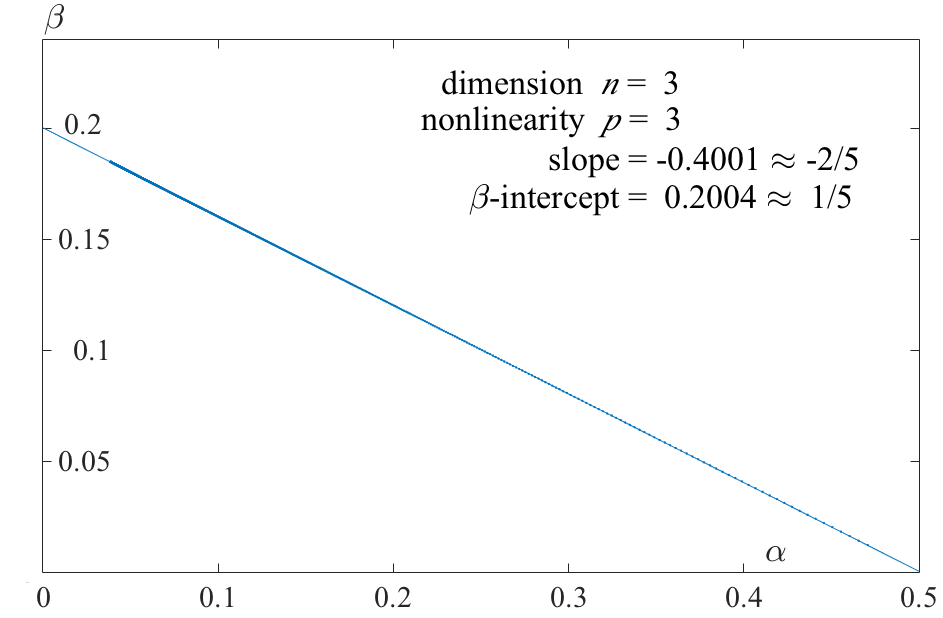}}
\scalebox{0.164}{\includegraphics{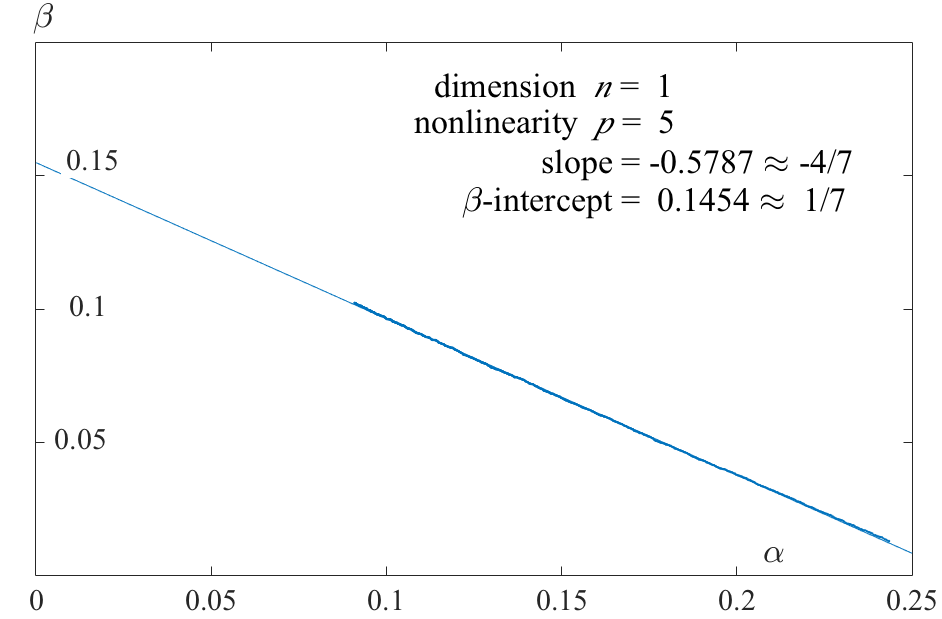}}
\hfill
\scalebox{0.164}{\includegraphics{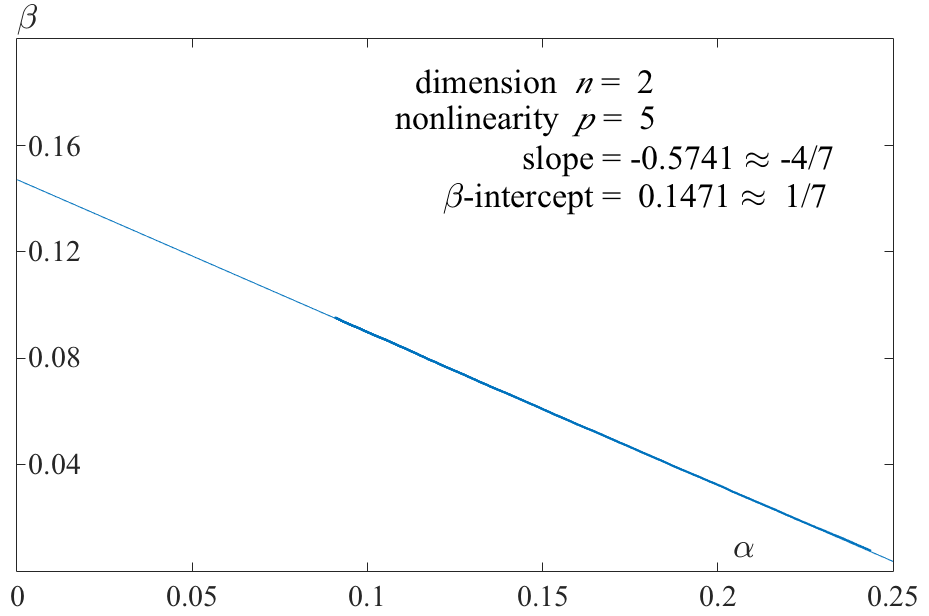}}
\hfill
\scalebox{0.164}{\includegraphics{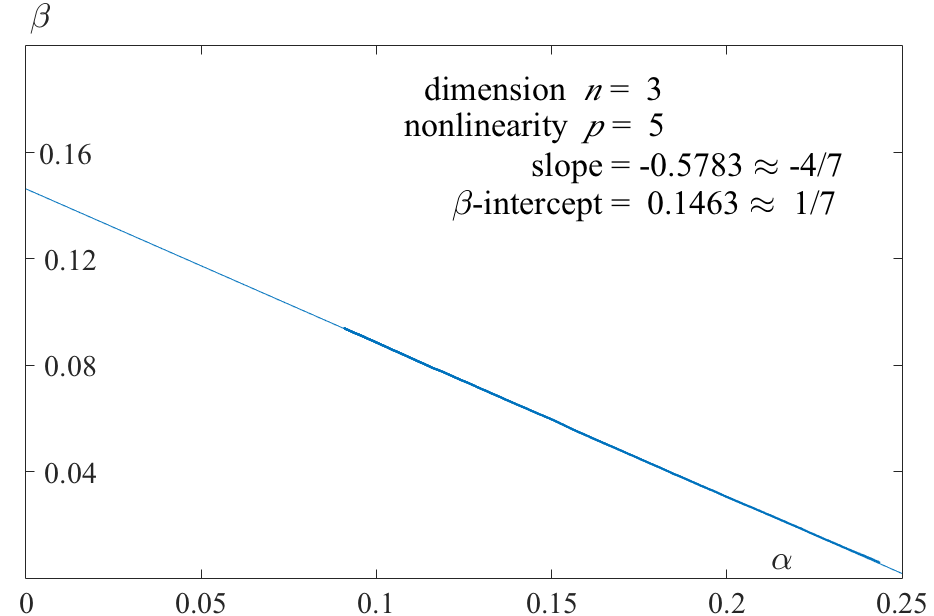}}
\scalebox{0.164}{\includegraphics{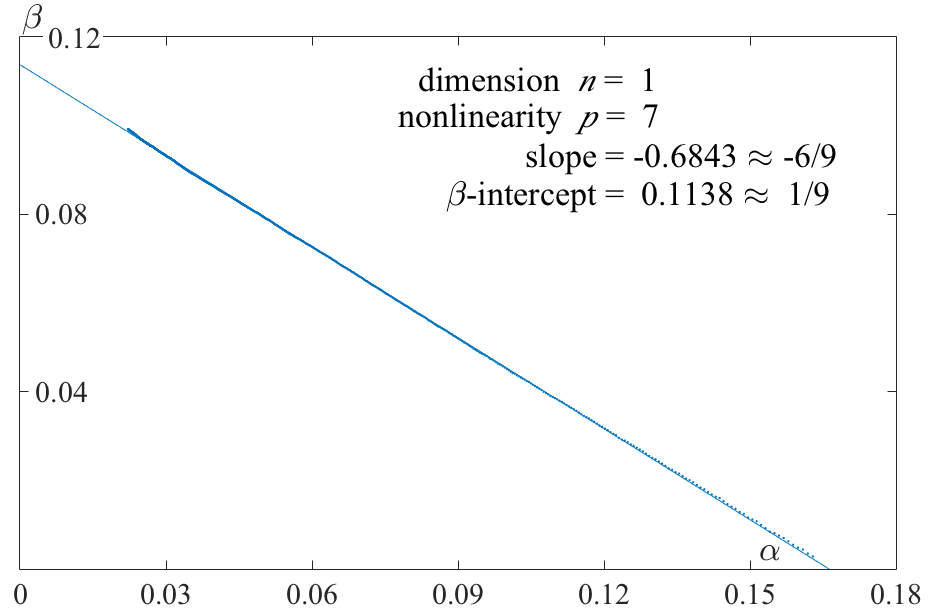}}
\hfill
\scalebox{0.164}{\includegraphics{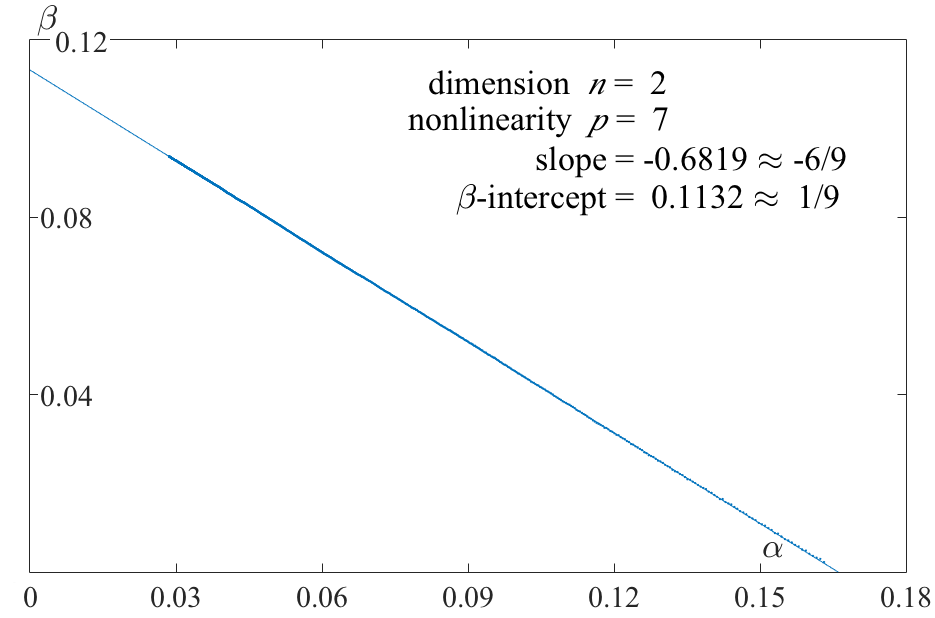}}
\hfill
\scalebox{0.164}{\includegraphics{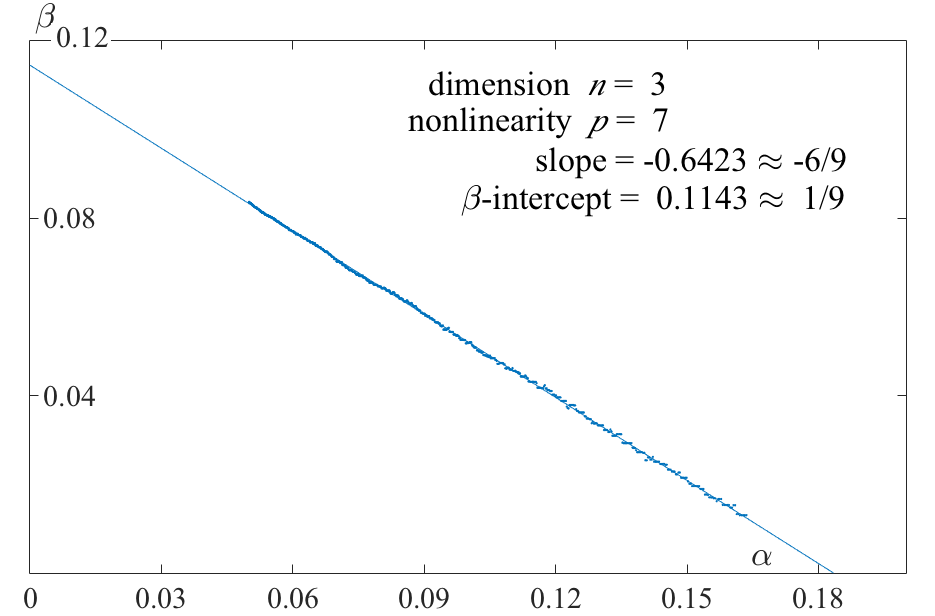}}
\caption{\small
Plot of data points $(\alpha,\beta)$ obtained numerically using Algorithm~A, and its linear regression, confirming the relation $\beta = (1-(p-1)\alpha)/(p+2)$ in Theorem~\ref{thm:ep}.  The initial data for the exciton-polariton systems (\ref{system0} nonlinear) and (\ref{systemB} linear) is $\epsilon^\alpha \exp(-\half|x|^2)$, and relative error $\epsilon$ between the nonlinear and linear systems is attained at time $t=C\epsilon^\beta$.
Notice that the $\alpha$-$\beta$ relation depends on the nonlinearity parameter $p$ but not on the spatial dimension $n$.  The norm of the solution is measured in $H^s(\RR^n)$ for $s=\lfloor n/2+1\rfloor$, the least integer greater than $n/2$.  Parameters are set at $g=1$, $\gamma=1$, and $\omega_0=1$.
}
\label{fig:Hs}
\end{figure}

\begin{figure}
\scalebox{0.164}{\includegraphics{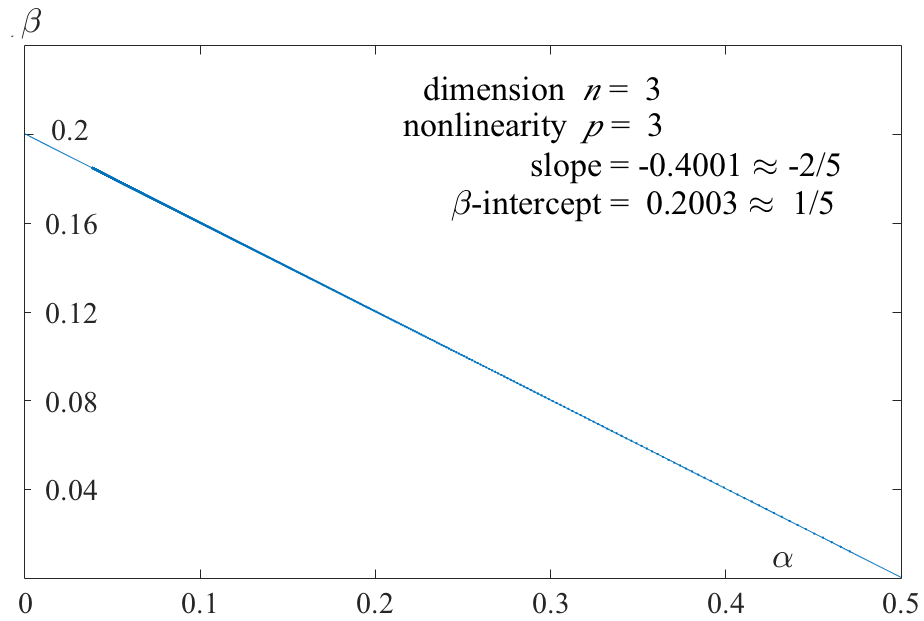}}
\hfill
\scalebox{0.164}{\includegraphics{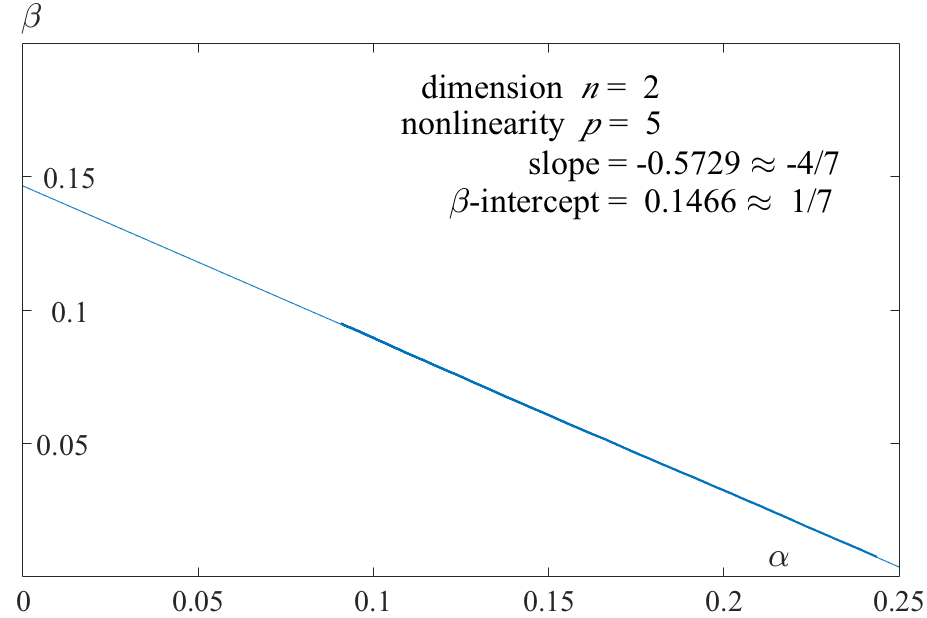}}
\hfill
\scalebox{0.164}{\includegraphics{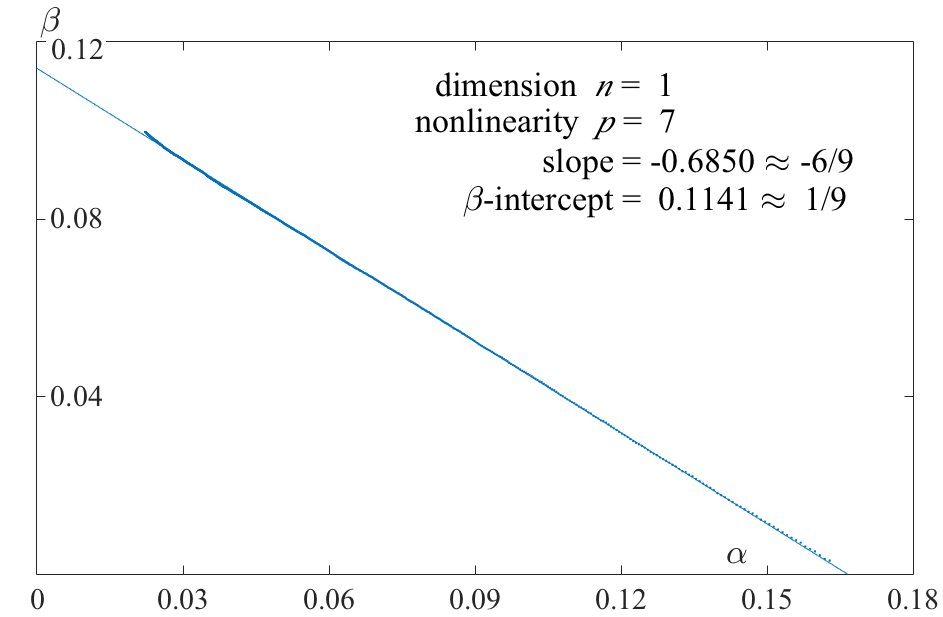}}
\caption{\small
Plot of data points $(\alpha,\beta)$ obtained numerically using Algorithm~A, and its linear regression, demonstrating the validity of the relation $\beta = (1-(p-1)\alpha)/(p+2)$ in Theorem~\ref{thm:ep} when the norm of the solution is measured in $L^2(\RR^n)$ instead of $H^s(\RR^n)$ with $s\!>\!n/2$.
The initial data for the exciton-polariton systems (\ref{system0} nonlinear) and (\ref{systemB} linear) is $\epsilon^\alpha \exp(-\half|x|^2)$, and relative error $\epsilon$ between the nonlinear and linear systems is attained at time $t=C\epsilon^\beta$.  Parameters are set at $g=1$, $\gamma=1$, and $\omega_0=1$.
}
\label{fig:L2}
\end{figure}

 Our EP code is base on the Rodrigo Platte's  NLS used in \cite{Holmer-Platte:2009aa} and \cite{Holmer-Perelman:2015aa} and  Svetlana Roudenko and Kai Yang'  nonlinear Klein-Gordon code used in  \cite{Roudenko-Yang:2017}.

\section{Concluding remarks}

This work is really a study of short-time delayed response to nonlinear effects in an evolution system when the nonlinearity is accessed through coupling to a hidden, not directly observed, variable.
With an initial condition of size~$\epsilon^\alpha$, nonlinear effects are negligible ($<\epsilon$) up to time $t=C\epsilon^\beta$, and the relation between the powers $\alpha$ and $\beta$~is
\begin{equation*}
  \beta = (1-(p-1)\alpha)/(p+2).
\end{equation*}
The factor of $1/(p+2)$ comes from the degree $p$ of the nonlinearity plus the coupling of $\phi$ to the hidden variable $\psi$ and the coupling from $\psi$ back to $\phi$.  We expect this principle to hold for a much wider class of nonlinear evolution systems.

Numerical computations show that the time $t=C\epsilon^\beta$ is sharp.  The theorem bounds the relative error by~$\epsilon$ at time $t=C\epsilon^\beta$, but the numerical computations invariably show that this bound is in fact attained at that time.  This remains to be proved.  The reason is evidently that, for short time, the size of the solution is accurately estimated by triangle (Minkowski) inequalities.  Estimates in the proofs do not take into account oscillations in the fields, which, for longer time would cause triangle inequalities to overestimate the size of solutions to the PDE.

Although the theorems are proved under the restriction that the Sobolev degree $s$ exceed $n/2$, numerical computations in $L^2(\RR^n)$, that is, $s\!=\!0$, demonstrate that this restriction should not be necessary.  The hypothesis $s\!>\!n/2$ is required only in the obtention of the constant $K_p$ in the bound
\begin{equation}
  \| |\psi|^p \|_{H^s} \,<\, K_p \| \psi \|_{H^s}^p\,,
\end{equation}
which is used in the inequalities (\ref{Kp1}) and~(\ref{Kp2}).  The algebra structure of $H^s(\RR^n)$ for $s\!>\!n/2$ underlies the proof of this bound.
The solutions treated in this paper are in fact continuous functions vanishing at infinity since $H^s(\RR^n)$ for $s\! >\! n/2$ is continuously embedded in $C_{0}(\RR^n)$ \cite[Theorem~3.2]{LinaresPonce2014}, thus  
$\|\phi(t)\|_{C_{0}(\RR^n)} \! \leq  K \|\phi(t)\|_{H^s(\RR^n)},$ for a constant  $K=K(n,s).$

\bigskip
\bigskip
\noindent
{\bfseries Acknowledgments.}
The authors would like to thank Rodrigo Platte, Svetlana Roudenko and Kai Yang for making their code available to us for the numerical simulations.
SPS acknowledges the support of NSF research grant DMS-1411393.


\end{document}